\documentclass[12pt]{amsart}               
\usepackage{graphicx,amsfonts,amssymb}
\usepackage[colorlinks,citecolor=blue,urlcolor=blue]{hyperref}

\theoremstyle{plain}
\newtheorem{thm}{Theorem}[section]
\newtheorem{lem}[thm]{Lemma}
\newtheorem{prop}[thm]{Proposition}

\theoremstyle{definition}

\newtheorem{conj}[thm]{Conjecture}
\newtheorem{exmp}[thm]{Example}

\theoremstyle{remark}

\newcommand{\cC}{\mathcal{C}}

\newcommand{\bi}{\mathbf{i}}

\newcommand{\R}{\mathbb{R}}

\newcommand{\rank}{\operatorname{rank}}

\title{Tensors of Nonnegative Rank Two}
\author[E. Allman et al.]{Elizabeth S. Allman, John A. Rhodes, \\ Bernd Sturmfels, and Piotr Zwiernik}
\address{Allman and Rhodes: University of Alaska at Fairbanks \hfill \break
Sturmfels and Zwiernik: University of California at Berkeley}
\email{e.allman@alaska.edu, j.rhodes@alaska.edu, \hfill \break
bernd@math.berkeley.edu, 
 pzwiernik@berkeley.edu}
\thanks{BS was supported by NSF
(DMS-0968882) and DARPA (HR0011-12-1-0011), and
PZ  by the European
Union 7th Framework Programme (PIOF-GA-2011-300975).
We thank Luke Oeding and Giorgio Ottaviani
for helpful conversations}

\begin{document}
\begin{abstract}
  A nonnegative tensor has nonnegative rank at most~2 if and only if
  it is supermodular and has flattening rank at most~2.  We prove this
  result, then explore the semialgebraic geometry of the general
  Markov model on phylogenetic trees with binary states, and comment
  on possible extensions to tensors of higher rank.
 \end{abstract}

\keywords {nonnegative tensor rank, latent class model, binary tree model}
\subjclass[2010]{15A69, 62H17, 14P10}

\maketitle

\section{Introduction}

This article offers a journey into {\em semialgebraic statistics}. By
this we mean the systematic study of statistical models as
semialgebraic sets.  We shall give a semialgebraic description of
binary latent class models in terms of binomials expressing
supermodularity, and we determine the algebraic boundary of this and
related models.  Our discussion is phrased in the language of
nonnegative tensor factorization \cite{de2008tensor, FriedlanderHatz}.

We consider real tensors $P=[p_{i_1 i_2 \cdots i_n}]$ of format $d_1
\times d_2 \times \cdots \times d_n$.  Throughout this paper we shall
assume that $n \geq 3$ and $d_1,d_2,\ldots,d_n \geq 2$.  Such a tensor
has {\em nonnegative rank at most $2$} if it can be written~as
\begin{equation}\label{eq:nrank}
P\quad = \quad a_1\otimes a_2\otimes \cdots \otimes a_n\,\,+\,\, 
b_1\otimes b_2\otimes\cdots \otimes  b_n,
\end{equation}
where the vectors $a_i,b_i \in \R^{d_i}$ are nonnegative for $i=1,2,\ldots,n$.
The set of such tensors is a closed semialgebraic subset
of dimension  $2(d_1+d_2 + \cdots + d_n)-2(n-1)$   in the 
tensor space  $\,\R^{d_1 \times d_2 \times \cdots \times d_n}$; see
\cite[\S 5.5]{JM}.
We present the following characterization of this semialgebraic set.

\begin{thm}
\label{theorem:eins}
A nonnegative tensor $P$ has nonnegative rank at most $2$ if and only if
$P$ is supermodular and has flattening rank at most~$2$.
\end{thm}

Here, {\em flattening} means picking any subset $A $ of
$[n]=\{1,2,\ldots,n\}$ with $1 \leq |A| \leq n-1$ and writing the
tensor $P$ as an ordinary matrix with $\prod_{i \in A}d_i$ rows and
$\prod_{j \not\in A } d_j$ columns.  The {\em flattening rank} of $P$
is the maximal rank of any of these matrices.  
Landsberg and Manivel \cite{LM} proved that
flattening rank $\leq 2$ is equivalent to  border rank $\leq 2$.

To define
supermodularity, we first fix a tuple $\pi =
(\pi_1,\pi_2,\ldots,\pi_n)$ where $\pi_i$ is a permutation of
$\{1,2,\ldots,d_i\}$.   Then $P$ is {\em  $\pi$-supermodular}~if
\begin{equation}
\label{eq:pisubm}
p_{i_1 i_2 \cdots i_n} \cdot\,
p_{j_1 j_2 \cdots j_n} \,\,\leq \,\,\,
p_{k_1 k_2 \cdots k_n} \cdot\,
p_{l_1 l_2 \cdots l_n}
\end{equation}
whenever
$\{i_r,j_r\} = \{k_r,l_r\}$ and
 $\pi_r(k_r) \leq \pi_r(l_r)$ holds for $r = 1,2,\ldots,n$.
We call a  tensor $P$ {\em supermodular} if it is $\pi$-supermodular for
some $\pi$.

Note that we are using multiplicative notation instead of
the additive notation more commonly used for supermodularity.
To be specific, if $d_1 = d_2 = \cdots = d_n=2$,
$\pi = ({\rm id}, {\rm id}, \ldots, {\rm id})$, and
$P$ is strictly positive, then $P$ being $\pi$-supermodular
means that $ {\rm log}(P)$ lies in the convex polyhedral cone
\cite[\S 4]{CRT} of
supermodular functions $\,2^{\{1,2,\ldots,n\}} \rightarrow \mathbb{R}$.

The set of $\pi$-supermodular nonnegative tensors $P$ of flattening rank $\leq 2$
is denoted $\mathcal{M}_\pi$ and called a {\em toric cell}.
The number of toric cells is $d_1 ! d_2 ! \cdots d_n !/2$.
Theorem \ref{theorem:eins} states that these cells stratify our model:
\begin{equation}
\label{eq:Nunion}
 \mathcal{M} \,\,= \,\,\,\cup_\pi \,\mathcal{M}_\pi .
 \end{equation}
The term {\em model}
refers to the fact that intersection of
(\ref{eq:Nunion}) with the probability simplex,
where all coordinates of $P$ sum to one,
is a widely used statistical model.
It is the mixture model for pairs of
independent distributions on
$n$ discrete random variables.
This justifies our earlier claims about
the dimension  of $\mathcal{M}$ and that it is topologically closed. 

Recall that the
 {\em Zariski closure} $\overline{\mathcal{S}}$ of a semialgebraic subset $\mathcal{S}$
of $\R^{N}$ is the zero set of all polynomials that vanish on $\mathcal{S}$. The
{\em boundary}  $\partial \mathcal{S}$ is the topological boundary of 
$\mathcal{S}$ inside $\overline{\mathcal{S}}$.
We define the
 {\em algebraic boundary} of $\mathcal{S}$ to be
 the Zariski closure $\overline{\partial \mathcal{S}}$ of its topological boundary.

Our second theorem concerns the algebraic boundaries of the model $\mathcal{M}$
and of toric cells $\mathcal{M}_\pi$. We regard these boundaries
as hypersurfaces inside the complex variety of
tensors of border rank $\leq 2$.
A {\em slice} of our tensor $P$ is a subtensor
of some format
$ d_1 {\times} {\cdots} {\times} d_{s-1} {\times} 1 {\times} d_{s+1} {\times} {\cdots} {\times} d_n $.
Subtensors of format
$ d_1 {\times} {\cdots} {\times} d_{s-1} {\times} 2 {\times} d_{s+1} {\times} {\cdots} {\times} d_n $ 
are  {\em double slices}.

\begin{thm} \label{theorem:zwei}
The algebraic boundary of $\mathcal{M}$ has
$ \sum_{i=1}^n d_i$ irreducible
components, given by  slices having
rank $\leq 1$. The algebraic boundary of
any toric cell  $\mathcal{M}_\pi$ has the same
irreducible components, but it  has $\sum_{i=1}^n
\binom{d_i}{2}$ further components, 
given by linearly dependent double slices.
\end{thm}

A double slice is {\em linearly dependent} if
its two slices are identical up to a multiplicative scalar.
In the second component count of Theorem \ref{theorem:zwei}
we exclude the special case $2 \times 2 \times 2$
because the ``further components''
fail to be hypersurfaces.
If $n=2$ then the rank $1$ constraint on slices is void, and
the algebraic boundary consists of the $d_1 d_2$ coordinate 
hyperplanes in  $\mathbb{R}^{d_1 \times d_2}$.
This is consistent with the fact
\cite[Example 4.1.2]{LAS} that
all nonnegative matrices of rank $2$ 
have nonnegative rank $2$.

This paper is organized as follows.  In Section 2 we derive our two
theorems for tensors of format $2 \times 2 \times 2$.  This extends
 results in \cite{allman2012semialgebraic,
altmannsquare, klaeretripod, pearl_tarsi86, pwz-2009-semialgebraictrees} 
on this widely studied latent class
model.  Here, our semialgebraic set $\mathcal{M}$ is full-dimensional
in $\R^{2 \times 2 \times 2}$, and it consists of four toric cells
that are glued together. Any two cells intersect along the locus where
one of the flattenings has rank one. The common intersection of all
cells is the {\em independence model} (tensors of rank $1$).  In
Section 3 we prove Theorems \ref{theorem:eins} and \ref{theorem:zwei}
for arbitrary $d_1,d_2,\ldots,d_n$.

The set $\mathcal{M}$ above appears in phylogenetics as the
{\em general Markov model on a star tree} with binary states.  Section
4 develops the extension of our results to phylogenetic trees other
than star trees. These models have another type of component in their
algebraic boundaries, characterized by the constraint that the ranks
of certain matrix flattenings of $P$ inconsistent with the tree
topology drop from $4$ to~$3$.  Supermodularity in this context was
pioneered by Steel and Faller \cite{steelfaller}.  Our results refine
earlier work on the general Markov model in \cite{allmanrhodesold,
  allman2012semialgebraic, klaeretripod, pwz-2009-semialgebraictrees}.

Section 5 concerns the challenges to be encountered when trying to
extend our results to tensors of higher rank. We present case studies
of algebraic boundaries for one identifiable model ($3 \times 3 \times
2$-tensors of rank $3$) and one non-identifiable model ($2 \times 2
\times 2 \times 2$-tensors of rank~$3$).

\section{The Base Case}

Let $P=[p_{ijk}]$ be a real $2\times 2\times 2$ tensor. 
Then $P$ has nonnegative rank at most $2$ if
there exist three nonnegative $2 \times 2$-matrices
$$
A_1=\left[\begin{array}{cc}
a_{11} & a_{12}\\
b_{11} & b_{12}
\end{array} \right] , \quad
A_2=\left[\begin{array}{cc}
a_{21} & a_{22}\\
b_{21} & b_{22}
\end{array} \right] \quad \hbox{and} \quad
A_3=\left[\begin{array}{cc}
a_{31} & a_{32}\\
b_{31} & b_{32}
\end{array} \right] 
$$
such that
\begin{equation}
\label{eq:para222a}
 \qquad p_{ijk} \, = \,a_{1i} a_{2j} a_{3k} + b_{1i} b_{2j} b_{3k} 
\qquad \hbox{for} \,\,\, i,j,k \in \{1,2\}. 
\end{equation}
For $\pi =  ({\rm id},{\rm id},{\rm id})$, the binomial inequalities for supermodularity are
\begin{equation}
\label{eq:super222}
\begin{matrix}
p_{111} p_{222} \geq p_{112} p_{221} && 
p_{111} p_{222} \geq p_{121} p_{212} &&
p_{111} p_{222} \geq p_{211} p_{122} \\ 
p_{112} p_{222} \geq p_{122} p_{212} &&
p_{121} p_{222} \geq p_{122} p_{221} &&
p_{211} p_{222} \geq p_{212} p_{221} \\
p_{111} p_{122} \geq p_{112} p_{121} &&
p_{111} p_{212} \geq p_{112} p_{211} &&
p_{111} p_{221} \geq p_{121} p_{211}
\end{matrix}
\end{equation}
Nonnegative $2 \times 2 \times 2$ tensors $P$ that  satisfy
these nine inequalities lie in the toric cell 
$\mathcal{M}_{{\rm id}, {\rm id}, {\rm id}} = \mathcal{M}_{(12),(12),(12)}$.
By label swapping $1 \leftrightarrow 2$, we obtain three other toric cells
$\mathcal{M}_{{\rm id}, {\rm id}, (12)} = \mathcal{M}_{(12),(12),{\rm id}}$,
$\mathcal{M}_{{\rm id}, (12), {\rm id}} = \mathcal{M}_{(12),{\rm id},(12)}$, and
$\mathcal{M}_{(12), {\rm id}, {\rm id}} = \mathcal{M}_{{\rm id},(12),(12)}$.
Thus, by definition, the semialgebraic set of all supermodular $2 \times 2 \times 2$-tensors is 
the union
\begin{equation}
\label{eq:unionoffour}
\mathcal{M}\,\,=\,\,
\mathcal{M}_{{\rm id}, {\rm id}, {\rm id}} \,\cup\,
\mathcal{M}_{{\rm id}, {\rm id}, (12)}  \,\cup \,
\mathcal{M}_{{\rm id}, (12), {\rm id}} \,\cup\,
\mathcal{M}_{(12), {\rm id}, {\rm id}} .
\end{equation}
Theorem \ref{theorem:eins} states that
$P \in \R^{2 \times 2 \times 2}$ has nonnegative rank at most $2$ if and only if $P$ lies in
$\mathcal{M}$.
We begin by proving the only-if direction.

\begin{lem}\label{lem:MthenSM}
If $P \in \R^{2 \times 2 \times 2}$ has nonnegative rank at most $2$,
then $P$ is supermodular. More precisely, 
define $\pi = (\pi_1,\pi_2,\pi_3)$ by $\pi_i = {\rm id}$  if  $\det (A_i)\geq 0$ 
and    $\pi_i = (12)$ if $\det (A_i)<0$.
     Then $P \in \mathcal{M}_{\pi}$.
         \end{lem}
    
\begin{proof} Let $P $ be as in (\ref{eq:para222a}). The last six constraints in (\ref{eq:super222}) 
specialize to
\begin{equation}
\label{eq:binomialfactors}
\begin{matrix}
 p_{112} p_{222} - p_{122} p_{212} & = & a_{32} b_{32} {\rm det}(A_1) {\rm det}(A_2), \\
 p_{121} p_{222} - p_{122} p_{221} & = & a_{22} b_{22} {\rm det}(A_1) {\rm det}(A_3), \\ 
 p_{211} p_{222} - p_{212} p_{221} & =  & a_{12} b_{12} {\rm det}(A_2) {\rm det}(A_3), \\
 p_{111} p_{122} - p_{112} p_{121}  & = & a_{11} b_{11} {\rm det}(A_2) {\rm det}(A_3), \\
 p_{111} p_{212} - p_{112} p_{211} & =  & a_{21} b_{21} {\rm det}(A_1) {\rm det}(A_3), \\ 
 p_{111} p_{221} - p_{121} p_{211} & = & a_{31} b_{31} {\rm det}(A_1) {\rm det}(A_2). 
 \end{matrix}
\end{equation}
First suppose that all $12$ parameters $a_{ij}$ and $b_{ij}$ 
and the three determinants
${\rm det}(A_k)  $ are positive. Then the six
expressions in (\ref{eq:binomialfactors}) are positive.
The first three constraints
in (\ref{eq:super222}) are also satisfied, as seen from
\begin{equation}
\label{eq:identitiessuchas}
\begin{matrix}
 p_{111} p_{222} - p_{121} p_{212} \,\,= \qquad \qquad \qquad \qquad \qquad \qquad 
 \qquad \qquad \\ 
( p_{111}( p_{112} p_{222} - p_{122} p_{212}) + 
                                 p_{212} (p_{111} p_{122} - p_{112} p_{121}) )/p_{112}.
\end{matrix}
\end{equation}
Second, consider all
tensors $P$ where the parameters 
$a_{ij} ,b_{ij}$ and determinants ${\rm det}(A_k)$ are nonnegative.
These lie in the closure of the previous case, so the nine binomials will be nonnegative.

Next observe that $\pi P=( p_{\pi_1^{-1}(i)\pi_2^{-1}(j)\pi_3^{-1}(k)})$ also
has nonnegative rank $\leq 2$,
  with parameterization given by swapping the columns of $A_i$ 
  whenever $\pi_i = (12)$.
  This changes the sign of $\det A_i$. Hence, $P$ is in
     $\mathcal{M}_\pi$
      if and only if $\pi^{-1}P$ is in $\mathcal{M}_{{\rm id},{\rm id},{\rm id}}$.
      \end{proof}

We now prove Theorem \ref{theorem:eins}
for $2\times 2\times 2$ tensors.
In this special case, the flattening rank is automatically $\leq 2$,
so there are no equational constraints, and our model
$\mathcal{M}$ is a full-dimensional subset of $\R^{2 \times 2 \times 2}$.

\begin{prop}\label{prop:main43} Let $P$ be a  nonnegative $2\times 2\times 2$-tensor. Then $P$
has nonnegative rank $\leq 2$ if and only if $P$ is supermodular.
\end{prop}

\begin{proof}
If $P$ has nonnegative rank $\leq 2$, then $P$ is supermodular by  Lemma \ref{lem:MthenSM}. 
For the converse, suppose that $P$ is supermodular. Define
\begin{equation}
\label{eq:UUUdef1}
\begin{array}{l}
U_{12}\,=\,p_{11+}p_{22+}-p_{12+}p_{21+},\\
U_{13}\,=\,p_{1+1}p_{2+2}-p_{1+2}p_{2+1},\\
U_{23}\,=\,p_{+11}p_{+22}-p_{+12}p_{+21},\\
\end{array}
\end{equation}
where a subscript $+$ refers to summing over all values of the given index.
For example, $p_{22+}=p_{221}+p_{222}$. Similarly, for $i=1,2$,  define
\begin{equation}
\label{eq:UUUdef2}
\begin{array}{l}
U_{12}^i\,=\,p_{11i}p_{22i}-p_{12i}p_{21i},\\
U_{13}^i\,=\,p_{1i1}p_{2i2}-p_{1i2}p_{2i1},\\
U_{23}^i\,=\,p_{i11}p_{i22}-p_{i12}p_{i21}.\\
\end{array}
\end{equation}
Our strategy is to first show that the following hold for~$P$:
\begin{itemize}
\item[(i)] $U_{12}U_{13}U_{23}\geq 0$,
\item[(ii)] $U_{ij}^1$ and $U_{ij}^2$ have the same sign as $U_{ij}$ for every $i<j$, and
\item[(iii)] if $U_{ij}=0$, then $U_{ij}^1=U_{ij}^2=0$.
\end{itemize}
Subsequently, in the second step, we will show that (i), (ii), (iii) imply that $P$ has the form
(\ref{eq:para222a}) with $A_1,A_2,A_3$ nonnegative. That second step
will follow proofs of closely related
results in \cite{allman2012semialgebraic,
altmannsquare, klaeretripod, pearl_tarsi86, pwz-2009-semialgebraictrees}.

Let $e = ({\rm id},{\rm id},{\rm id})$. Since $\pi \mathcal{M}_e = \mathcal{M}_\pi$,
and the conditions (i),(ii),(iii) are invariant under label swapping, it suffices to consider
$P \in \mathcal{M}_e$.
By definition of $e$--supermodularity, $U_{ij}^1,U_{ij}^2\geq 0$. We need to show that 
$U_{ij}\geq 0$ also. By symmetry it suffices to show that $U_{12}\geq 0$. We have
\begin{equation}\label{eq:U12}
U_{12}=U_{12}^1+U_{12}^2+(p_{111}p_{222}+p_{221}p_{112}-p_{211}p_{122}-p_{121}p_{212}).
\end{equation}
We show that the expression in parentheses  is nonnegative for 
$P \in \mathcal{M}_e$. We write this expression as
 $R=f_{11}+f_{22}-f_{12}-f_{21}$, where
 $$
f_{11}=p_{111}p_{222},\quad f_{12}=p_{121}p_{212},\quad  f_{21}=p_{211}p_{122},\quad f_{22}=p_{221}p_{112}
$$
Note that, by \eqref{eq:super222}, we have $f_{11} \ge \max \{ f_{12}, f_{21} \}$. This implies
$R \ge 0$ if either $p_{ijk}  = 0$ for some $i,j,k$, or if $f_{22}\geq \min\{f_{12},f_{21}\}$.
Thus, we assume that $f_{ij} > 0$ and $f_{22} < \min \{f_{12}, f_{21}\}$.
The supermodular inequalities $p_{121}p_{211}\leq p_{111}p_{221}$ and 
$p_{212}p_{122}\leq p_{222}p_{112}$
imply
$$ f_{12} f_{21} \,\,= \,\,
p_{121}p_{212}p_{211}p_{122}\,\,
\leq \,\, p_{111}p_{222}p_{221}p_{112} \,\,=\,\, f_{11} f_{22}.
$$
Hence $[f_{ij}]$ is supermodular itself. As a consequence, we have
$$ \frac{f_{21}}{f_{11}}-1
\,\leq \,\frac{f_{22}}{f_{12}}-1
\,\leq \,\bigg(\frac{f_{22}}{f_{12}}-1\bigg)\frac{f_{12}}{f_{11}},
$$
where the second inequality holds since $f_{22} < f_{12} \le f_{11}$.

After multiplying both sides by $f_{11}$ we obtain
$$
f_{21}-f_{11}\,\leq \,f_{22}-f_{12}
$$
or equivalently $R\geq 0$. It follows that $U_{12}\geq 0$ and, by 
symmetry, that $U_{ij}\geq 0$ for all $i<j$; therefore (i) and (ii) hold. 
The identity (\ref{eq:U12}) and the inequality $R\geq 0$
together imply that (iii) holds as well.

\smallskip

We take up separately the cases where the product $U_{12}U_{13}U_{23}$ is positive 
or is zero.  Suppose first that $U_{12}U_{13}U_{23} > 0$. 
A special role in our argument will be played by the {\em hyperdeterminant},
$$\begin{array}{l}
{\rm Det}(P)=4p_{111}p_{122}p_{212}p_{221} + 4p_{112}p_{121}p_{211}p_{222}
        + p_{111}^2p_{222}^2 + p_{122}^2p_{211}^2 \\ + p_{112}^2p_{221}^2 
              -   2p_{111}p_{112}p_{221}p_{222} - 2p_{111}p_{121}p_{212}p_{222}
    - 2p_{111}p_{122}p_{211}p_{222}
   \\   + p_{121}^2p_{212}^2
  -  2p_{112}p_{121}p_{212}p_{221}
    - 2p_{112}p_{122}p_{211}p_{221 }
   - 2p_{121}p_{122}p_{211}p_{212}.\end{array}
$$
One can verify the identity
\begin{equation}
\label{eq:hyperdetidentity}
p_{+++}^2{\rm Det}(P)\,\,=\,\,\mu^2+4U_{12}U_{13}U_{23}, \qquad {\rm where} 
\end{equation}
$$
\mu=p_{+++}^2p_{222}-p_{+++}(p_{2++}p_{+22}+p_{+2+}p_{2+2}+p_{++2}p_{22+})+2p_{2++}p_{+2+}p_{++2}.
$$
Then (\ref{eq:hyperdetidentity}) implies
${\rm Det}(P)>0$. 

By \cite[Proposition 5.9]{de2008tensor} 
we can write $P$ in terms of \emph{real}  vectors $a_i,b_i$ as in (\ref{eq:para222a}). We obtain 
the identities
$$\begin{array}{lcl}
U_{12}&=&\det A_1 \det A_2\,(a_{31}+a_{32})(b_{31}+b_{32}),\\
U_{13}&=&\det A_1 \det A_3\,(a_{21}+a_{22})(b_{21}+b_{22}),\\
U_{23}&=&\det A_2 \det A_3\,(a_{11}+a_{12})(b_{11}+b_{12}).
\end{array}
$$  
Since $U_{12}U_{13}U_{23}$ is strictly positive,  the coordinate sum of each 
vector  $a_i, b_i$ is nonzero. Hence  our model can be equivalently parametrized by 
\begin{equation}\label{eq:nrank2P3}
P=s a_1\otimes a_2\otimes a_3+tb_1\otimes b_2\otimes b_3, 
\end{equation}
where $s,t\in \R$ and the coordinates of $a_i,b_i$ sum to~$1$. We now show 
that (i)--(iii) ensures these parameters to be nonnegative. 
Note that
\begin{equation}\label{eq:idents3}
\begin{array}{lcl}
U_{12}&=&\det A_1 \det A_2\,st,\\
U_{12}^1&=& \det A_1 \det A_2\, st a_{31}b_{31},\\
U_{12}^2&=& \det A_1 \det A_2\, st  a_{32}b_{32},
\end{array}
\end{equation}
and similar formulas hold for $U_{13},U_{13}^1,U_{13}^2$ and $U_{23},U_{23}^1,U_{23}^2$.

Under the specialization (\ref{eq:nrank2P3}),
the hyperdeterminant factors as
$$
{\rm Det}(P)\,\,=\,\,(st\det A_1 \det A_2\det A_3)^2.
$$
This gives
$$
st\quad=\quad\frac{U_{12}U_{13}U_{23}}{{\rm Det}(P)}\quad > \quad 0,
$$
and thus either $s,t > 0$ or $s,t < 0$. 
By (ii), $U_{12},U_{12}^1,U_{12}^2$ have the same signs. Hence 
$a_{31}b_{31} \ge  0$ and $a_{32}b_{32} \ge 0$  by (\ref{eq:idents3}).
 This, together with the fact that $[p_{++i}]=s a_3+t b_3$ is a nonnegative vector, 
 implies that $a_3,b_3\in \R^2_{\ge 0}$ if $s,t > 0$ and $a_3,b_3\in \R^2_{\le 0}$ if $s,t < 0$. 
 The same argument shows that $a_1,b_1,a_2,b_2\in \R^2_{\ge 0}$ if $s,t > 0$ and 
 $a_1,b_1,a_2,b_2\in \R^2_{\le 0}$ if $s,t < 0$. Hence,  we obtain a nonnegative 
 decomposition in (\ref{eq:nrank2P3}).
 
Suppose now that $U_{12}U_{13}U_{23}=0$. Without loss of generality, 
assume $U_{12}=0$.  Hypothesis (iii) implies $U_{12}^1=U_{12}^2=0$.
Regard the expressions in (\ref{eq:UUUdef1})
and (\ref{eq:UUUdef2}) as elements in the
polynomial ring $\mathbb{Q}[p_{111},p_{112},\ldots,p_{222}]$.
A computation reveals the prime decomposition
\begin{multline*}
 \langle U_{12}, U_{12}^1, U_{12}^2 \rangle\, \,\,=\ \ 
\langle \,\hbox{$2 {\times} 2$-minors of } {\rm Flat}_{1|23}(p) \,\rangle \\ \cap 
\langle \,\hbox{$2 {\times} 2$-minors of } {\rm Flat}_{2|13}(p) \,\rangle ,
\end{multline*}
where
$ \,{\rm Flat}_{1|23}(p)= \begin{pmatrix}
p_{111} \!&\! p_{112} \! &\! p_{121} \! &\! p_{122}\\
p_{211} \!&\! p_{212} \! &\! p_{221}\! &\! p_{222}
\end{pmatrix}$, and similarly for 
${\rm Flat}_{2|13}(p)$.

\smallskip
\noindent Hence one of these two flattenings of the
tensor $P \in \mathcal{M}_e$ has rank~$1$. Suppose it is the first.
We can find $v \in \mathbb{R}^2_{\geq 0}$ such that
 $p_{ijk}= v_i \cdot p_{+jk}$ for every $i,j,k\in \{1,2\}$.
 Since the $ 2 \times 2$-matrix $(p_{+jk})$ 
 can be written as  $\,(p_{+jk}) =a_2\otimes a_3+b_2\otimes b_3\,$
 for some  $a_2,b_2,a_3,b_3\in \R_{\geq 0}^2$, we
obtain the desired nonnegative representation
(\ref{eq:nrank2P3}) by setting $a_{1i}=b_{1i}=v_i$.
\end{proof}

\smallskip

Theorem \ref{theorem:zwei} tells us that the algebraic boundary
of $\mathcal{M}$ equals
\begin{multline*}
 \{p_{112} p_{222} {=} p_{122} p_{212} \} \cup
\{p_{121} p_{222} {=} p_{122} p_{221} \} \cup
\{p_{211} p_{222} {=} p_{212} p_{221} \} \,\cup  \\
\{p_{111} p_{122} {=} p_{112} p_{121} \} \cup
\{p_{111} p_{212} {=} p_{112} p_{211} \} \cup
\{p_{111} p_{221} {=} p_{121} p_{211}\}. \phantom{\cup}
\end{multline*}
Each toric cell $\mathcal{M}_\pi$ has exactly the same
algebraic boundary because the linear dependence constraint on double slices is void 
in the $2 \times 2 \times 2$-case. 
The coordinate planes $\{p_{ijk} = 0\}$ 
are {\bf not} part of the algebraic boundary of $\mathcal{M}$
or $\mathcal{M}_\pi$. Indeed, the inverse image of
$\{p_{ijk} = 0\}$ under the parametrization 
\label{eq:para222} lies in the boundary. But, if any of the parameters 
$a_{ij}$ or $b_{ij}$ is zero then the tensor $P$ has a rank $1$ slice.
Hence, the set $\{p_{ijk} = 0\} \cap \mathcal{M}$ 
lies in the union above.
Similarly, the hyperdeterminant 
$\{{\rm Det}(P) = 0\}$ is {\bf not} a component in the algebraic boundary of $\mathcal{M}$.

\begin{exmp}
It is instructive to look at a $3$-dimensional picture of
our $7$-dimensional model $\mathcal{M}$.
We consider the {\em Jukes-Cantor slice} given~by
$$
\begin{bmatrix}
 p_{111} & p_{112} \\
 p_{121} & p_{122} 
 \end{bmatrix} 
 \, = \,
 \begin{bmatrix}
 x & y \\
 z & w 
 \end{bmatrix}
 \quad \hbox{and} \quad
 \begin{bmatrix}
 p_{211} & p_{212} \\
 p_{221} & p_{222} 
 \end{bmatrix}
 \, = \,
 \begin{bmatrix}
 w & z \\
 y & x 
 \end{bmatrix}.
$$ 
Under this specialization, the hyperdeterminant factors as
 \begin{equation}
 \label{eq:hyperdetspecial}
  {\rm Det}(P) \,= \,
 (x+y+z+w)(x+y-z-w)(x-y+z-w)(x-y-z+w) . 
 \end{equation}
 Consider the tetrahedron $\,
 \bigl\{ (x,y,z,w) \in \R^4_{\geq 0} \,:\,
 x+y+z+w = 1/2 \bigr\}$.
 Fixing the signs of the last three factors in (\ref{eq:hyperdetspecial})
 divides the tetrahedron
 into four bipyramids and four smaller tetrahedra. Inside our slice,
the four toric cells of (\ref{eq:unionoffour})
occupy the bipyramids. Each toric cell
is precisely the object in \cite[Figure 1]{EHS}.
Redrawn on the right in Figure 1, 
its convex hull is the bipyramid,
and it contains  six of the nine edges.
Any two of the toric cells meet in a line segment
such as 
$\,\{ x+y-z-w = x-y+z-w = 0 ,\, x-y-z+w \geq 0 \}$.
The algebraic boundary of each toric cell
consists of the same three quadrics
$\{xy=zw\} $, $\{xz=yw\}$ and $\{xw=yz\}$.
Neither the three planes in (\ref{eq:hyperdetspecial})
nor the four facet planes of the tetrahedron are in 
 the algebraic boundary. 
\qed
\end{exmp}

\begin{figure}\label{fig:slice}
\includegraphics[scale=.26,trim=100 20 10 10, clip=true]{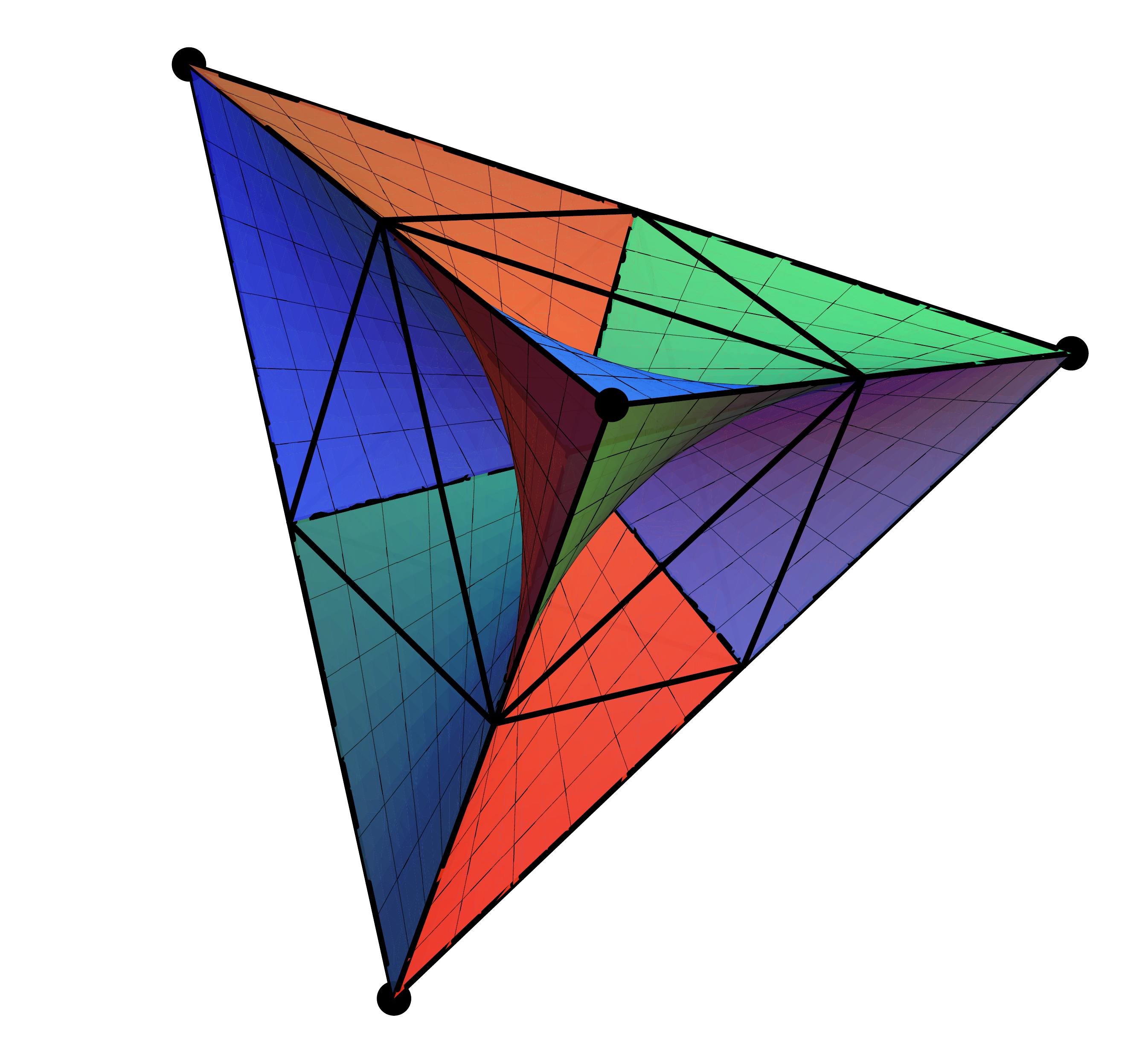} \!\!\!\!\!\!
\includegraphics[scale=.25,trim=220 220 180 0, clip=true]{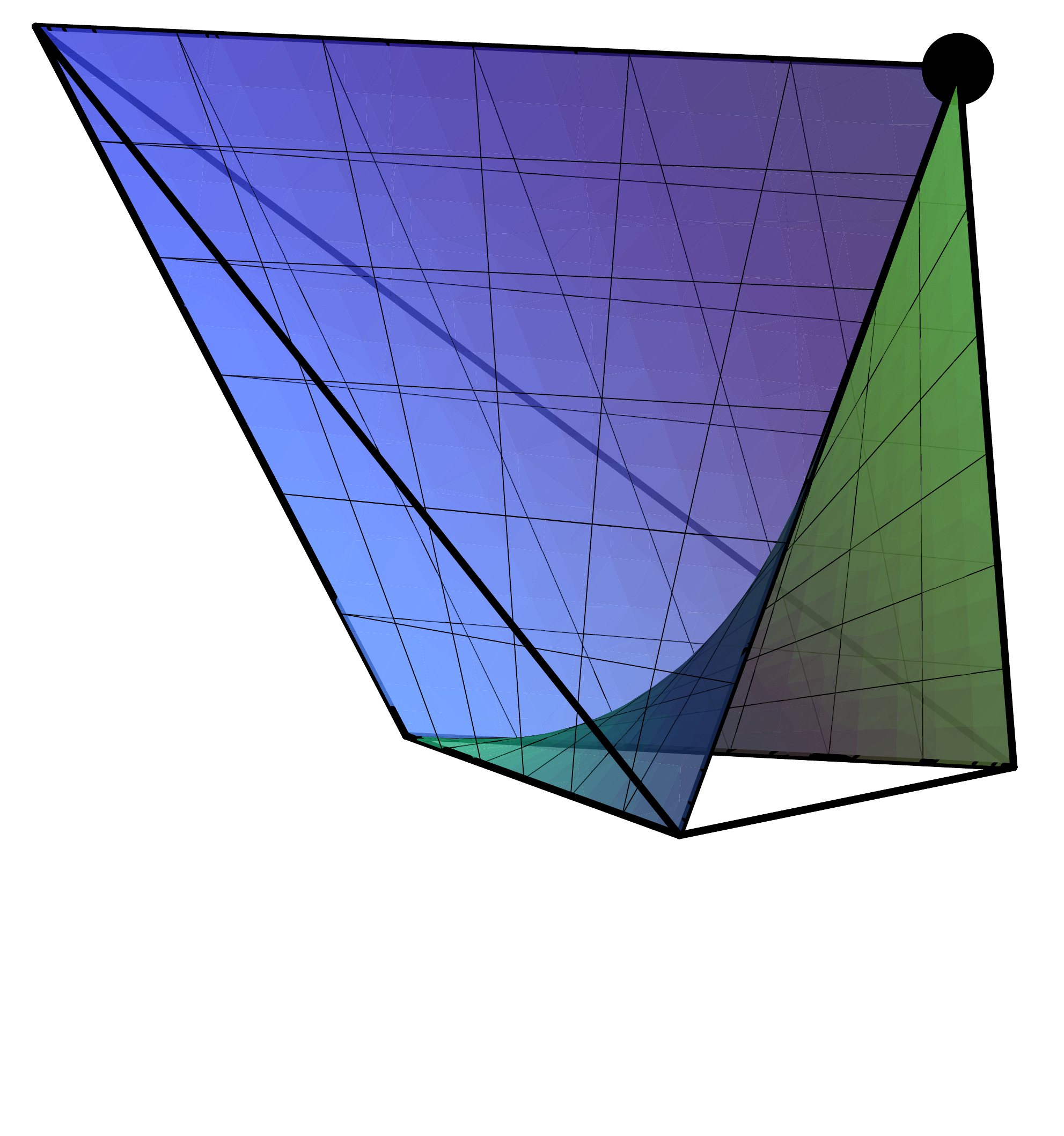}
\caption{Jukes-Cantor slice showing $2\times 2\times 2$ tensors of
non-negative rank $\leq
2$. Each toric cell is bounded by three quadrics and contains a vertex
of the tetrahedron.}
\end{figure}

\section{The General Case}

Before embarking on the general proofs of Theorems \ref{theorem:eins}
and \ref{theorem:zwei}, let us briefly go over an example that
exhibits the general behavior.

\begin{exmp} \label{ex:333eins} Consider the semialgebraic set
  $\mathcal{M}$ of $3 \times 3 \times 3$-tensors of nonnegative rank
  $\leq 2$. The Zariski closure $\overline{\mathcal{M}}$ of
  $\mathcal{M}$ in $\R_{\geq 0}^{3 \times 3 \times 3}$ has dimension
  $14$ and is defined by $222$ cubic equations \cite[Table 3]{GSS},
  namely $3 \times 3$-minors of the $3 \times 9$-matrices ${\rm
    Flat}_{1|23}(P)$, ${\rm Flat}_{2|13}(P)$ and ${\rm
    Flat}_{3|13}(P)$.  The model $\mathcal{M}$ decomposes into $108$
  toric cells $\mathcal{M}_\pi$, each defined in
  $\overline{\mathcal{M}}$ by $162$ quadratic binomial inequalities of
  the form (\ref{eq:pisubm}).

  A quick way to generate these inequalities, for $\pi = ({\rm id},
  {\rm id}, {\rm id})$, is to run the following code in the computer
  algebra system {\tt Macaulay2} \cite{M2}:
\begin{small}
\begin{verbatim}
R = QQ[p111,p112,p113,p121,p122,p123,p131,p132,p133,
       p211,p212,p213,p221,p222,p223,p231,p232,p233,
       p311,p312,p313,p321,p322,p323,p331,p332,p333];
S = QQ[a1,a2,a3,b1,b2,b3,c1,c2,c3];
f=map(S,R,{a1*b1*c1,a1*b1*c2,a1*b1*c3,a1*b2*c1,a1*b2*c2,a1*b2*c3,
a1*b3*c1,a1*b3*c2,a1*b3*c3,a2*b1*c1,a2*b1*c2,a2*b1*c3,a2*b2*c1,
a2*b2*c2,a2*b2*c3,a2*b3*c1,a2*b3*c2,a2*b3*c3,a3*b1*c1,a3*b1*c2,
a3*b1*c3,a3*b2*c1,a3*b2*c2,a3*b2*c3,a3*b3*c1,a3*b3*c2,a3*b3*c3});
gens gb kernel f
\end{verbatim}
\end{small}
Being  $\pi$-supermodular means that each of the binomials in the
resulting Gr\"obner basis, such as {\tt p223*p312-p212*p323}, must be non-positive.

The algebraic boundary of $\mathcal{M}$ has nine irreducible
components, corresponding to the nine slices of $P$. It is instructive
to see how our $162$ hypersurfaces, like $ \{ p_{223} p_{312} =
p_{212} p_{323} \} \cap \overline{\mathcal{M}}$, break into these
components.  Each individual toric cell $\mathcal{M}_\pi$ has $18$
irreducible components in its algebraic boundary: now also the $9$
double-slices kick in. The intersection of all $108$ toric cells is
the {\em Segre variety} of rank $1$ tensors, whose reverse
lexicographic Gr\"obner basis we identified with
(\ref{eq:pisubm}). \qed
\end{exmp}

A {\em marginalization} of $P$ is any tensor obtained from $P$ by
summing all slices for some fixed indices.   For instance,
the $2 {\times} 2$-matrix $(p_{ij+})$ is a marginalization of the
$2 {\times} 2 {\times} 2$-tensor $P = (p_{ijk})$.  
The following
lemma, whose proof is delayed, will be useful in proving Theorem
\ref{theorem:eins}.  
  
\begin{lem}\label{lem:margins}
All marginalizations of a supermodular tensor are supermodular,
and ditto for $e$-supermodular with
$e = ({\rm id}, {\rm id},\ldots,  {\rm id})$. 
In addition, all flattenings of a supermodular tensor are supermodular.
\end{lem}

In this lemma, and in the remainder of the paper, we use the term {\em flattening} to 
include all tensor flattenings, not just the matrix flattenings described in the introduction.
We now prove our first main result.

\begin{proof}[Proof of Theorem \ref{theorem:eins}]
  Suppose first that $P$ has nonnegative rank $\leq 2$.  Then
  $P$ has the form (\ref{eq:nrank}) with
  $a_i,b_i \in \R_{\geq 0}^{d_i}$. As tensor rank cannot increase
  under flattening, we conclude that $P$ has flattening rank $\leq 2$.  
  
   Consider the
 $d_i \times 2$-matrix with columns $a_i,b_i$.  By swapping
  rows we can make all $2 \times 2$-subdeterminants of these $n$
  matrices $(a_i,b_i)$ nonnegative. But swapping rows in these matrices corresponds to acting on $P$ by
  $\pi$, where $\pi P :=[p_{\pi^{-1}(\bi)}]$ for
  $\bi=(i_1,\ldots,i_n)$ and
  $\pi^{-1}(\bi)=(\pi^{-1}_1(i_1),\ldots,\pi^{-1}_n(i_n))$. Since $P
  \in \mathcal{M}_e$ if and only if $\pi P\in \mathcal{M}_\pi$, it
  suffices to prove the following: if $P$ has the form (\ref{eq:nrank})
  with $a_{ik}b_{il}\geq a_{il}b_{ik}$ for every $i$ and all $k\leq l$
  then $P \in \mathcal{M}_e$.
 
 To prove this we define an auxiliary  
 $2{\times} d_1 {\times} \cdots {\times} d_n$ tensor $\hat P$ by
$$
\hat p_{1 i_1 \cdots  i_n}=a_{1i_1} a_{2 i_2}\cdots a_{ni_n}
\quad \hbox{and} \quad \hat p_{2 i_1 \cdots i_n}=b_{1i_1} b_{2 i_2}\cdots b_{ni_n}.
$$
We claim that $\hat P$ is $e$-supermodular. For this, we  need to check that
\begin{equation}
\label{eq:needtocheck}
\hat p_{i_0i_1\cdots i_n}\hat p_{j_0j_1\cdots j_n}\quad \leq\quad \hat p_{k_0k_1\cdots k_n}\hat p_{l_0l_1\cdots l_n},
\end{equation}
for all $i_r, j_r$ such that $k_r=\min\{i_r,j_r\}$, $l_r=\max\{i_r,j_r\}$. 
This holds with equality if $i_0=j_0$. If $i_0\neq j_0$ we have two cases to consider,
and our claim (\ref{eq:needtocheck}) is equivalent to
the inequality
$$
\max\{a_{1j_1}b_{1i_1}\cdots
a_{nj_n}b_{ni_n},\,\,a_{1i_1}b_{1j_1}\cdots
a_{ni_n}b_{nj_n}\} \,\,\, \leq\,\,\,
 a_{1k_1}b_{1l_1}\,\cdots\,a_{nk_n}b_{nl_n}.
$$
Our assumption on the
$2 \times 2$-subdeterminants of $(a_r,b_r)$ ensures
$$\max\{a_{rj_r}b_{ri_r},a_{ri_r}b_{rj_r}\}\leq  a_{rk_r}b_{rl_r}\qquad\mbox{for every }r\in [n].$$
This gives the desired inequality, and therefore $\hat P$ is $e$-supermodular.    
But $P$ is a marginalization of $\hat P$ because
 $p_{i_1 \cdots i_n}=\hat p_{1 i_1 \cdots i_n}+\hat p_{2 i_1 \cdots i_n}$,
so Lemma \ref{lem:margins} then implies that $P$ is $e$-supermodular.

For the converse, consider any supermodular $d_1\times \cdots \times d_n$ tensor $P$ of
 flattening rank $\leq 2$.
Let $F_i$ be the flattening of $P$ given by the partition $\{i\}, [n]\backslash \{i\}$. 
Suppose $\rank (F_i) <2$ for some $i$, say $i=1$.
Then $P = v\otimes P'$ for some 
$v \in \R_{\ge0}^{d_1}$ and  $P' = [p_{+i_2\cdots i_n}]$. 
By Lemma \ref{lem:margins},
the marginalization
 $P'$ is supermodular with flattening rank $\le 2$.
By repeated application of this argument, we may reduce to tensors $P$ whose $d_i \times (d_1\cdots d_{i-1} d_{i+1} \cdots d_n)$-flattenings
 $F_i$ all have rank exactly $2$. 

We next reduce to  tensors of format $2\times\cdots\times 2$. Let
 $P$ be a supermodular $d_1\times \cdots \times d_n$ tensor
 all of whose flattenings are of rank $2$, and $L_i \subseteq \R^{d_i}$ the span of the
 columns of a flattening $F_i$. Two suitable columns of $F_i$ give a
 nonnegative basis $\{a ,b\}$ of $L_i$. We modify this
 basis to $\{a',b'\}$ so that, after permuting
 entries, it is nonnegative and
 \begin{align*}
a' &= (1,0,*,\ldots,*), \\
b' &= (0,1,*,\ldots,*).
\end{align*}
To obtain this nonnegative basis first set $a''=a-tb$, using the maximal 
$t$ for which $a''$ is
nonnegative.  Then set $b''=b-sa''$ with the
maximal $s$ for which $b''$ is nonnegative. The vectors
$a'',b''$ each have an entry of 0 in a position where
the other does not. Rescaling so the non-zero entries in these
positions become 1, and permuting entries to bring these positions to
the first two, we obtain the desired $a',b'$.

Now every column of $F_i$ is in the nonnegative span of 
$a', b'$.  More concretely, we have $\,F_i=C_i ^T \cdot F_i'$,
where $C_i$ has rows $a',b'$, and $ F_i'$ is the
first two rows of $F_i$.  
On tensors, this is expressed by
$$P\,=\,P'*_iC_i,$$ 
where $P'$ is the double slice of $P$ with
$i$th index in $\{1,2\}$ and 
$P'*_iC_i$ denotes the linear action of $C_i$ on the
$i$th index of $P'$.  
Applying this construction in each index we
find (after suitable relabelings) that
\begin{equation}P\,=\,P_0 * (C_1,\ldots, C_n),\label{eq:kkkto222}
\end{equation}
where $P_0$ is the $2{\times} \cdots{\times} 2$ subtensor of $P$
obtained by restricting all indices to $\{1,2\}$, and the $2\times
d_i$-matrices $C_i$ are real and nonnegative.

Our hypotheses ensure that $P_0$ is supermodular with all flattening ranks
$2$. Moreover, if $P_0$ has nonnegative rank $2$, then it follows from
equation \eqref{eq:kkkto222} that $P$ also has nonnegative rank
$2$. Explicitly, if $P_0=a_1\otimes \cdots \otimes a_n+b_1\otimes
\cdots\otimes b_n$ is a nonnegative decomposition, then $\,P=\tilde
a_1\otimes \cdots\otimes \tilde a_n+\tilde b_1\otimes \cdots\otimes
\tilde b_n\,$ with $\,\tilde a_i=a_i C_i$, $\tilde b_i=b_i C_i$
nonnegative.

It remains to show the result for $2\times \cdots\times 2$ tensors.
Let $P'$ denote the $2\times2\times 2^{n-2}$ flattening of $P$ from
the tripartition $\{1\},\{2\},[n]\backslash \{1,2\}$. By Lemma
\ref{lem:margins}, $P'$ is supermodular.  By Proposition
\ref{prop:main43}, each $2\times 2 \times 2$ subtensor of $P'$ has
nonnegative rank $\leq 2$. 
The argument of the last three paragraphs  implies that $P'$
itself has nonnegative rank $\leq 2$, so 
 $$P'\,\,=\,\,a_1\otimes a_2\otimes a_3\,+\,b_1\otimes b_2\otimes b_3,$$
 with $a_1,a_2,b_1,b_2 \in \R_{\ge0}^2$, $a_3,b_3\in \R_{\ge0}^{2^{n-2}}$.
 The matrices  $A = (a_1,b_1)^T$ and $B=(a_2,b_2)^T$ are invertible,
by our assumptions on the $2\times 2^{n-1}$ flattening ranks of $P$.
Acting on the tensor $P$ by their inverses, we get
$$ \tilde P\,= \,e_1\otimes e_1\otimes N_1+e_2\otimes e_2\otimes N_2,$$
where $N_1,N_2$ are nonnegative tensors whose vector flattenings are $a_3,b_3$.

Consider any bipartition $A, B$ of $\{3,\ldots,n\}$. The $2^{|A|+1}\times
2^{|B|+1}$ flattening of $\tilde P$ using the bipartition $\{1\}{\cup}
A,\,\{2\} {\cup} B$ is block-diagonal, with blocks given by $A|B$
flattenings of $N_1,N_2$. This $2^{|A|+1}\times 2^{|B|+1}$
matrix has rank $\leq 2$, so either both flattenings of $N_i$ have rank
$\leq 1$, or one $N_i$ is zero. But $N_i=0$ is impossible since that
would mean some $2\times 2^{n-1}$ flattening of $P$ has rank
$1$. Hence the $A|B$ flattenings of $N_1,N_2$ have rank 1. Since $A,B$
were arbitrary, both $N_i$ have (nonnegative) rank
$1$. Consequently, $\tilde P$ has nonnegative rank $2$, and so does~$P$.
\end{proof}

It remains to prove Lemma \ref{lem:margins}. We shall use the
{\em Four Function Theorem} of Ahlswede and Daykin \cite{daykin78},
here presented in a special case:

\begin{prop}\label{th:4fun}{\rm [Ahlswede-Daykin]}
Fix $n\geq 2$ and a nonnegative  $d_1{\times} \cdots{\times} d_n$-tensor
  $P=[p_{i_1\ldots i_n}]$. For any collection $\cC$ of indices $\bi=(i_1,\ldots,i_n)$
   in $[d_1]\times \cdots\times [d_n]$ define $p_{\cC}=\sum_{\bi\in \cC}p_\bi$. 
Suppose that
\begin{equation}\label{eq:ADcond}
p_\bi \cdot p_\mathbf{j} \,\,\leq \,\,p_{\bi\vee \mathbf{j}} \cdot p_{\bi \wedge \mathbf{j}}
\quad
\hbox{for any two indices $\bi,\mathbf{j}$},
\end{equation}
where $\vee$, $\wedge$ are join and meet operations that gives
$[d_1]\times\cdots\times [d_n]$ a lattice structure.  Then for any two
collections $\cC, \cC'$, we have
$$
p_\cC \cdot p_{\cC'} \,\,\leq \,\, p_{\cC\vee \cC'} \cdot p_{\cC\wedge\cC'},
$$
where 
$\,
\cC\vee\cC'=\{\bi\vee \mathbf{j}:\, \bi\in \cC, \mathbf{j}\in \cC'\}\,$ and $\,\,
\cC\wedge\cC'=\{\bi\wedge \mathbf{j}:\, \bi\in \cC, \mathbf{j}\in \cC'\}$.
\end{prop}

\begin{proof}[Proof of Lemma \ref{lem:margins}]Let $P$ be a
  supermodular $d_1{\times} \cdots {\times} d_n$-tensor. For the first
  assertion, it suffices to show that $P'=[p_{+i_2\cdots i_n}]$ is
  supermodular. The general statement for marginal tensors follows by
  induction.
 
  If $P$ is $\pi$-supermodular, define the lattice structure on
  $[d_1]\times \cdots\times [d_n]$ by taking $\mathbf k=\mathbf
  i\wedge \mathbf j$ if and only if $\pi(\mathbf k)$ is the
  coordinatewise minimum of $\pi(\mathbf i)$ and $\pi(\mathbf j)$.
  Similarly, $\mathbf l=\mathbf i\vee \mathbf j$ if and only if
  $\pi(\mathbf l)$ is the coordinatewise maximum of $\pi(\mathbf i)$
  and $\pi(\mathbf j)$.  Fix $\,\bi',\mathbf j'\in [d_2]\times
  \cdots\times [d_n]\,$ and set
$$
\cC=\{(i_1,\bi'):\, i_1\in [d_1]\},\qquad \cC'=\{(i_1,\mathbf j'):\, i_1\in [d_1]\}.
$$
We have $p_{\cC}=\sum_{\bi\in \cC}p_\bi=p_{+\bi'}$ and $p_{\cC'}=p_{+\mathbf j'}$.
The tensor $\pi P=(p_{\pi^{-1}(\bi)})$ is $e$-supermodular. 
Proposition \ref{th:4fun} now gives
$$ p_{+\mathbf i'} \cdot p_{+\mathbf j'}
\,\,\leq \,\,p_{+(\mathbf i'\wedge \mathbf j')} \cdot p_{+(\mathbf i'\vee \mathbf j')}.$$
This means that
  $P'$ is $\pi'$-supermodular, where $\pi'=(\pi_2,\ldots,\pi_n)$.
  
  We now prove that every flattening of $P$ is supermodular. Let
  $Q=[q_{\alpha_1\cdots \alpha_r}]$ be a flattening of $P$
  corresponding to the partition $A_1,\ldots,A_r$ of
  $\{1,\ldots,n\}$. Let $h_i=\prod_{j\in A_i}d_j$, then
  $\boldsymbol\alpha=(\alpha_1,\ldots,\alpha_r)\in [h_1]\times
  \cdots\times [h_r]$. Without loss of generality we can assume that
  $\alpha_i$ indexes elements of $\prod_{j\in A_i} [d_j]$ ordered
  lexicographically. Every $q_{\boldsymbol\alpha}$ is equal to $p_\bi$
  for some $\bi$, so that each $\boldsymbol\alpha$ corresponds to a
  unique $\bi$. Since $P$ is supermodular, there exists
  $\pi=(\pi_1,\ldots,\pi_n)$ such that for every $\bi, \mathbf{j}$ we
  have $p_\bi p_\mathbf{j}\leq p_{\bi\wedge \mathbf{j}}p_{\bi\vee
    \mathbf{j}}$, where $\bi\wedge \mathbf{j}$ and $\bi\vee
  \mathbf{j}$ is as defined in the previous paragraph.
  
   Define now $\boldsymbol\alpha
  \wedge \boldsymbol\beta$ and $\boldsymbol\alpha \vee
  \boldsymbol\beta$ to be the $r$-tuples corresponding to $\bi\wedge
  \mathbf{j}$ and $\bi\vee \mathbf{j}$. The permutation $\pi$ induces the
  corresponding $r$-tuple of permutations
  $\tilde\pi=(\tilde\pi_1,\ldots,\tilde\pi_r)$ such that
  $\pi(\bi)=\tilde\pi(\boldsymbol\alpha)$. By construction, we have
  $p_{\boldsymbol\alpha}p_{\boldsymbol\beta}\leq p_{\boldsymbol\alpha
    \wedge \boldsymbol\beta}p_{\boldsymbol\alpha \vee
    \boldsymbol\beta}$, where $\tilde{\pi}(\boldsymbol\alpha \vee
  \boldsymbol\beta)\leq \tilde\pi(\boldsymbol\alpha \vee
  \boldsymbol\beta)$. This implies that $Q$ is $\tilde\pi$-supermodular.
\end{proof}

We now prove the second theorem stated in the Introduction.

\begin{proof}[Proof of Theorem \ref{theorem:zwei}]
The formula (\ref{eq:nrank}) defines a polynomial map
$$\phi :  \R^{2(d_1+d_2+\cdots+d_n)}  \rightarrow
\R^{d_1 \times d_2 \times \cdots \times d_n} $$ such that $\mathcal{M}
= \phi\bigl(\R^{2(d_1+\cdots d_n)}_{\geq 0} \bigr)$ is the set of
tensors of nonnegative rank $\leq 2$.  We modify the domain by assuming the coordinate sums of all $a_i$ and $b_i$  are $1$, while adding two mixture parameters $s,t$ as in
(\ref{eq:nrank2P3}). This does not change the image, but makes the map generically 2-to-1.
More specifically, $\phi$ is 2-to-1 on the
open set where $st\ne0$ and each pair $a_i,b_i$ is linearly independent.  Since this open set intersects the coordinate
hyperplane $\{a_{ij} = 0\}$ (or $\{b_{ij} = 0\}$), the map $\phi$ is generically finite on that hyperplane.
Hence the closure of the image $\phi(\{a_{ij} = 0\})$ is an
irreducible subvariety of codimension $1$ in $\overline{\mathcal{M}}$.
Moreover, in any neighborhood of a point on $\{a_{ij} = 0\}$ there are
points with $a_{ij}<0$ that are not mapped into the interior of $\mathcal{M}$.
Indeed,
generically the fiber containing such a point only contains its image
under label swapping, and thus all points in the fiber have a negative
coordinate.  Thus $\phi(\{a_{ij} = 0\})$ is a component of the
algebraic boundary of $\mathcal{M}$.

By restricting to open subsets $\mathcal{U}_\pi$ where the signs of
all $2 \times 2$-minors of the matrices $(a_i,b_i)$ are fixed, we see
that $\overline{\phi(\{a_{ij}=0\})}$ is also a component in the
algebraic boundary of $\mathcal{M}_\pi$.  Additional pieces of the
boundary of $\mathcal{U}_\pi$ are the quadrics $\{a_{ij} b_{ik} =
a_{ik} b_{ij}\}$, on whose general points the map $\phi$ is also
2-to-1.  Therefore the varieties $\overline{\phi(\{a_{ij}b_{ik} =
  a_{ik} b_{ij} \})}$ are irreducible of codimension $1$ in
$\overline{\mathcal{M}_\pi}$, and, by the same argument as above, they
are also components of the algebraic boundary of $\mathcal{M}_\pi$.

We next argue that there are no components in the algebraic boundary of
$\mathcal{M}$ or $\mathcal{M}_\pi$ other than the two types we just
identified.  This follows from Theorem \ref{theorem:eins}. 
Let $P \in \partial \mathcal{M}_\pi$.
Consider the binomials $\, p_{i_1 i_2 \cdots i_n}  p_{j_1 j_2 \cdots j_n} \,-\,
p_{k_1 k_2 \cdots k_n}  p_{l_1 l_2 \cdots l_n}\,$ 
that correspond to facets of the polyhedral cone of supermodular functions. 
For such a facet binomial,
the indices in the four appearing unknowns $p_{\bullet}$ agree in all but two of the positions.
All other binomials (\ref{eq:pisubm}) admit representations such as
(\ref{eq:identitiessuchas}).
 The expansion of a facet binomial into parameters $a_{ij}, b_{ij}$
factors into coordinates and $2 \times 2$-determinants
as in (\ref{eq:binomialfactors}).
Hence, 
at the two points in $\phi^{-1}(P)$,
one of these factors must vanish, and this implies
that $P$ lies on one of the hypersurfaces we already identified above.

We finally identify $\overline{\phi(\{a_{ij} = 0\})}$
and $\overline{\phi(\{a_{ij}b_{ik} = a_{ik} b_{ij} \})}$
with the rank loci described in the statement of Theorem \ref{theorem:zwei}.
If the coordinate $a_{ij} $ vanishes then the
$j$-th slice of $P$ in the $i$-th dimension drops its rank
from $\leq 2$ to $\leq 1$.
Likewise, if $a_{ij} b_{ik} = a_{ik} b_{ij}$, then
the $j$th and $k$th slices of $P$  in dimension $i$ becomes
linearly dependent. Hence the irreducible components of
the algebraic boundaries of $\mathcal{M}$ and $\mathcal{M}_\pi$
are uniquely characterized by lying in the following two types of rank loci:
\begin{enumerate}
\item[(a)] the variety of tensors $P$ of border rank $\leq 2$ such that
a particular slice has border rank $\leq 1$;
\item[(b)] the variety of tensors $P$ of border rank $\leq 2$
such that a particular double slice
is linearly dependent. 
\end{enumerate}
This completes the proof of Theorem \ref{theorem:zwei}.
\end{proof}

We believe that the rank loci in (a) and (b) are irreducible varieties,
and that their prime ideals are generated by the relevant
subdeterminants of format $2 \times 2$ and $3 \times 3$.
At present we do not know how to prove this.
A similar issue for tree models appears in 
Conjecture \ref{conj:tree2}.
For the case of Example \ref{ex:333eins}, we proved
irreducibility by computation:

\begin{exmp} \label{ex:333zwei} The variety $\overline{\mathcal{M}}$
  of $3 \times 3 \times 3$ tensors of border rank $\leq 2$ has
  dimension $14$ and degree $783$. Using {\tt Macaulay2} \cite{M2}, we
  verified that both (a) and (b) define irreducible subvarieties of
  dimension $13$.  The variety (a) has degree $882$, and its prime
  ideal is minimally generated by $9$ quadrics and $187$ cubics. The
  variety (b) has degree $342$, and its prime ideal is minimally
  generated by $36$ quadrics and $90$ cubics. All ideal generators can
  be chosen from the relevant subdeterminants. \qed
\end{exmp}

One may ask how efficiently the model membership can be tested. The number of facets of the submodular cone is a
polynomial in the size of the tensor, and each facet
inequality involves precisely four of the unknowns.
Hence supermodularity for positive tensors can be tested  in
polynomial time. For instance, a $2\times 2\times \cdots \times 2$-tensor has $N = 2^n$ cell
entries, and the facets correspond to the $2$-faces of the $n$-cube (see the proof of Theorem \ref{theorem:zwei}), of which
there are only $n(n-1)2^{n-3} = O(N^{1+\epsilon})$.

\section{Binary Tree Models}
In this section we study the extension of our results to 
the {\em general Markov model} $\mathcal{M}_T$
 on a phylogenetic tree $T$ with binary states
\cite{allmanrhodesold, allman2012semialgebraic, draismakuttler,klaeretripod,
pwz-2009-semialgebraictrees}.
The special case when  $T$ is a {\em star tree},
with only one internal node,  corresponds to
$2 \times 2 \times \cdots \times 2$-tensors of nonnegative rank $\leq 2$.
For arbitrary trees $T$,
Steel and Faller \cite{steelfaller} showed that
distributions in $\mathcal{M}_T$ are supermodular,
by a marginalization argument as in Lemma~\ref{lem:margins}.

We assume that $T$ has $n \geq 3$ leaves, 
 $E$ is the set of edges of $T$, and
one of the $|E|-n+1$ internal nodes is the root of $T$.
We specify each probability distribution $P$ in the model $\mathcal{M}_T$
 by a nonnegative root distribution $\pi\in
\R_{\ge0}^2$, together with a $2\times2$ Markov matrix $M_e$ for
each edge $e$, directed away from the root. The entries of
$\pi$ and of each row of each $M_e$ sum to $1$. 
These choices determine a point $\theta = \bigl(\pi,(M_e)_{e \in E}\bigr)$ in 
the cube $\,\Theta = [0,1]^{2|E|+1}$.
That cube serves as the domain for the model parametrization 
$\, \phi:\Theta \twoheadrightarrow \mathcal{M}_T \,\subset\, \R_{\geq 0}^{2 \times 2 \times \cdots \times 2}$.
This can be found in explicit form in \cite[Equation (1)]{allmanrhodesold}.
The map $\phi$ is locally identifiable. To be precise, each general fiber consists
of $2^{|E|-n+1}$ points, corresponding to label swapping on the internal nodes.
Hence our binary tree model $\mathcal{M}_T =\phi(\Theta)$
is a compact semialgebraic set of dimension $2|E|+1$
inside the probability simplex of dimension $2^n-1$.
It is known that $\mathcal{M}_T$ is independent of the
choice of the root node.

The prime ideal that defines the Zariski closure
$\overline{\mathcal{M}_T}$ is known.  It is generated by the $3 \times
3$-minors of all flattenings of $P$ that are compatible with
$T$. Here, a split $(A,A^c)$ of $[n]$ is {\em compatible}
with $T$ if the intersection of any path between two leaves in $A$
with any path between two leaves in $A^c$ is either empty or just one
internal node.  This was first shown set-theoretically for trivalent
trees by Allman-Rhodes \cite{allmanrhodesold}.  The ideal-theoretic
statement for arbitrary trees $T$ is seen by combining the result of
Draisma-Kuttler in \cite{draismakuttler} with the result of Raicu in
\cite{raicu}.

Our main result of this section concerns the algebraic boundary of the general Markov model
$\mathcal{M}_T$ inside the phylogenetic variety  
$\overline{\mathcal{M}_T}$.

\begin{thm}\label{thm:tree1}
The algebraic boundary of the binary 
tree model $\mathcal{M}_T$ has
$n+|E|$ irreducible components,
two for each of the $n$ pendant edges,
and one for each of the $|E|-n$ internal edges.
The components
are closures of images of facets of the cube $\Theta$, 
as described below.
\end{thm}

The components of the algebraic boundary of~$\mathcal{M}_T$ are as follows:
\begin{enumerate}
\item For each pendant edge $e$ with leaf $\ell$, fix one row of the
$2\times 2$-matrix $M_e$. 
(The other row gives the same two components.)
Nonnegativity of either entry
determines a facet $F$ of the cube $\Theta$. 
Then $\overline{\phi(F)}$ is a component. It has the following
{\em equational description} inside $\mathcal{M}_T$. 
If the internal node on $e$ is $r$-valent, it
gives a partition  $(L_1=\{\ell\},L_2,\ldots,L_r)$ of $[n]$.
  Flatten $P$ accordingly to a
      $2 \times 2^{|L_2|}\times 
  \cdots \times 2^{|L_r|}$ tensor.
    The rank of the 
  $1 \times 2^{|L_2|}\times 
  \cdots \times 2^{|L_r|}$ slice selected by $F$
 drops to~$\leq 1$ on $\overline{\phi(F)}$.
\item For each internal edge $e$, fix any one entry of the
$2\times 2$-matrix $M_e$. 
(The other three entries 
give the same component.)
 Nonnegativity of that entry
determines a facet $F$ of the cube~$\Theta$.
Then $\overline{\phi(F)}$ is a component. It has the following
{\em equational description} inside $\mathcal{M}_T$. 
 Let $T[e]$ be the tree obtained from $T$ by 
contracting $e$. For either matrix flattening of $P$
that is compatible with $T[e]$ but not with $T$,
the rank drops to $\leq 3$ on~$\overline{\phi(F)}$.
  \end{enumerate}

At present we do not know whether the  equational descriptions above
(in terms of tensor rank)
are enough to cut out the 
codimension $1$ subvarieties
$\overline{\phi(F)} $ of $\mathcal{M}_T$.
For this, it would suffice to prove the following:

\begin{conj}\label{conj:tree2} 
The rank varieties in (1)~and~(2)  are
irreducible.
\end{conj}

We have a computational proof of Conjecture \ref{conj:tree2}
in the smallest non-trivial case, 
the trivalent tree on $4$ taxa, which we discuss next.
 
 \begin{exmp} \label{ex:12|34}
 Let $n=4$ and $T$ the trivalent tree with split $12|34$.
The phylogenetic variety lives in $\mathbb{P}^{15}$
and it has dimension $11$:
\begin{equation}
\label{eq:ex12|34A}
\overline{\mathcal{M}_T} \,= \,
\biggl\{ \,P \,\in \mathbb{P}^{15} \,: \, {\rm rank}
\begin{bmatrix}
p_{1111} & p_{1112} & p_{1121} & p_{1122} \\
p_{1211} & p_{1212} & p_{1221} & p_{1222} \\
p_{2111} & p_{2112} & p_{2121} & p_{2122} \\
p_{2211} & p_{2212} & p_{2221} & p_{2222} 
\end{bmatrix} \leq 2 \biggr\}.
\end{equation}
The model $\mathcal M_T$ is composed of 
eight $11$-dimensional cells $\mathcal M_\pi$, where $\pi=(\pi_1,\pi_2,\pi_3,\pi_4)$.
As before,
 $\mathcal M_{\rm id,\rm id,\rm id,\rm id}=\mathcal M_{(12),(12),(12),(12)}$.
These cells are glued together along lower-dimensional
models corresponding to forests obtained by deleting edges of the tree. 
For instance $\mathcal M_{\rm id,\rm id,\rm id,\rm id}$ is glued to 
$\mathcal M_{\rm id,\rm id,(12),(12)}$ along the model of two independent 
2-leaf trees. It is also glued to $\mathcal M_{\rm id,\rm id,\rm id,(12)}$ 
along a model of a $3$-leaf tree and an independent leaf, and similarly to 
3 other cells. Finally, it is glued to the remaining cells
$\mathcal M_{(12),\rm id,(12),\rm id}$ and $\mathcal M_{(12),\rm id,\rm id,(12)}$ 
along even more degenerate models, of a forest with one 2-leaf tree and two 
singleton leaves. All eight cells intersect in the model of four independent leaves. 
The various strata correspond to
$\mathbb{P}^3 {\times} \mathbb{P}^3$,
$\mathbb{P}^7 {\times} \mathbb{P}^1$,
$\mathbb{P}^3 {\times} (\mathbb{P}^1)^2$ and
$(\mathbb{P}^1)^4$.

The algebraic boundary of $\mathcal{M}_T$ 
has eight irreducible
 components of type (1), such as
\begin{equation}
\label{eq:ex12|34B}
\bigl\{ \,P \,\in \,\overline{\mathcal{M}_T} \,: \, {\rm rank}
\begin{bmatrix}
p_{1111} & p_{1112} & p_{1121} & p_{1122} \\
p_{1211} & p_{1212} & p_{1221} & p_{1222} 
\end{bmatrix}  \leq 1 \bigr\}. 
\end{equation}
The $2 {\times} 2$-minors of (\ref{eq:ex12|34B}) 
and $3 {\times} 3$-minors of (\ref{eq:ex12|34A})
generate a prime~ideal.

The ninth component of $\overline{\partial \mathcal{M}_T}$ comes from the internal edge
and is of type (2).
 It is defined by the $4 \times 4$-determinant of
either of the two flattenings other than (\ref{eq:ex12|34A}).
These two determinants are equal and irreducible on $\overline{\mathcal{M}_T}$,
so they give the prime ideal of that component.
\qed
 \end{exmp}
 
To prove Theorem \ref{thm:tree1}, we consider the singular locus
 $\Theta_{\rm sing}$ of the parametrization $\phi$. By definition,
 $\Theta_{\rm sing}$   is the closed subset of
the cube $\Theta$ where the rank of the Jacobian matrix
of $\phi$ drops to $2 |E|$ or below.
 
 \begin{lem}\label{lem:sing}
   $\Theta_{\rm sing}$ is the subset of points in $\Theta$ where
   either the root distribution $\pi$ has a zero entry, or
   some Markov matrix $M_e$ is singular.
 \end{lem}
 
 \begin{proof}  
   A tree $T_n$ with $n$ leaves is obtained by attaching a cherry to a
   leaf $\ell$ of an $(n-1)$-leaf tree $T_{n-1}$. Assuming the
   matrices $M_e$ on $T_{n-1}$ are non-singular and $\pi$ has non-zero
   entries, then the distribution for $T_{n-1}$ flattens on the edge
   incident to $\ell$ to a $2 \times 2^{n-2}$ matrix $A$ of rank
   $2$. Let $a_1, b_1$ be the rows of $A$, and $a_2, b_2$ and
   $a_3,b_3$ the rows of
   the
   matrix parameters on the edges of the cherry. Then the distribution for $T_n$,
   appropriately flattened, is $ a_1\otimes a_2\otimes a_3+b_1\otimes
   b_2\otimes b_3.$
   
   \smallskip
   
We next show the map  
$$\psi:(a_1,b_1,a_2,b_2,a_3,b_3)\mapsto a_1\otimes
   a_2\otimes a_3+b_1\otimes b_2\otimes b_3$$
where the entries of $a_2,b_2,a_3,b_3$ sum to 1,  
is non-singular precisely at points where all pairs $a_i,b_i$
are linearly independent.  That $\psi$ is singular at points where some pair
$a_i,b_i$ is dependent is straightforward. To show the rest of this claim,
we allow arbitrary real entries in the vectors, to take advantage of a group action.

Let $G$ be the subgroup of $GL(2^{n-2})\times GL(2)\times GL(2)$
consisting of matrix triples
$(g_1,g_2,g_3)$ where the rows of $g_2$ and $g_3$ sum to $1$.
The group $G$ acts on both the domain and range of $\psi$, and
intertwines as
$$\psi(zg)=\psi(z)g,\ g\in G.$$
Hence the Jacobian matrix of $\psi$ has constant rank on each
orbit. But the orbit of any point with all pairs $a_i,b_i$ linearly
independent is dense in the domain. Thus if $\psi$ were singular at
such a point, it would be singular everywhere. Since $\psi$ is
generically 2-to-1, that is impossible.  

Note that the statement of the lemma for the 3-leaf~tree follows from the previous paragraphs.
Building the tree $T_n$ inductively from $T_{n-1}$
writes the Jacobian of $\psi$ as a product of block matrices
of smaller Jacobians. From this we see that
 $\Theta_{\rm sing}$ consists of points where either $\pi$ has a zero entry, or
 some $M_e$ is singular.
   \end{proof}
 
   \begin{lem} \label{lem:fiber} If $\theta\in \Theta_{\rm sing}$, then
     the fiber of $\theta$ intersects the boundary of $\Theta$,
     i.e.~there exists $\theta'\in \partial\Theta$ with
     $\phi(\theta')=\phi(\theta)$. Moreover, $\theta'$ can be found in
     a facet of $\Theta$ where some entry of a Markov matrix is zero.
 \end{lem}
 
\begin{proof} 
  For a $3$-leaf tree, rooted at the internal node, consider the parameters
  $\theta=(\pi, M_1, M_2, M_3)\in\Theta_{\rm sing}$. If $\pi_i=0$, then
  we may replace row $i$ of any or all $M_j$ with  $(1,0)$ to obtain $\theta'$.
  Otherwise, suppose $M_3$ is singular yet there are no zeros in the parameters.
  Define $\theta'$ by
  $\pi'=\pi M_1$, $M_1'= \hbox{the identity matrix}$,
  $M_2'=\operatorname{diag}(\pi')^{-1} M_1^T\operatorname{diag}(\pi)M_2$,
  and $M_3'=M_3$. 
  One checks that $\phi(\theta') = \phi(\theta)$,
  and $M_1'$ has a zero entry.
  
  The result is derived inductively for larger trees, by viewing them as built from
  3-leaf trees by attaching cherries.
\end{proof}

\begin{proof}[Proof of Theorem \ref{thm:tree1}]
  Points in $\operatorname{Int}(\Theta) \backslash \Theta_{\rm sing}$
  must map to points in the relative interior of the model $\mathcal{M}_T$. Thus the
  boundary of $\mathcal{M}_T$ is a subset of  $\phi(\partial \Theta) \cup
  \phi(\Theta_{\rm sing})$. By Lemma
  \ref{lem:fiber}, this is contained in $\phi(\partial  \Theta)$.

  To see that each of the components listed is a boundary component,
  we must show they have codimension 1 in the model, and a Zariski
  dense subset of points in them are limits of points outside the
  model. Since the complement of $\Theta_{\rm sing}$ intersects these
  facets of $\Theta$ in non-empty open sets, the
  codimension is as needed. Since all elements of a fiber of the
  parameterization $\phi$ which contains non-singular points are related by
  label swapping, even when the map is extended outside $\Theta$, one
  sees that non-singular points outside $\Theta$ cannot be mapped into
  the model, yet they are mapped arbitrarily close to the claimed
  component.

We have discussed all but two of the $4|E|+2$ facets of $\Theta$.
The remaining two facets,  where an entry of
the root distribution $\pi$ is $0$, contain only
  elements of $\Theta_{\rm sing}$, by Lemma \ref{lem:sing}.
  By Lemma \ref{lem:fiber}, they
  lie in fibers with points where
  some entry of a Markov matrix is zero. Thus they are mapped into a
  component of the boundary already identified.
  
  \smallskip
  
  It remains to be shown that the equational descriptions
  given in (1) and (2) are valid  on the respective
   components of the boundary of $\mathcal{M}_T$.
  
   For a  pendant edge $e$ as in (1), we can assume that the root of the tree is located at
  the non-leaf end of $e$. The sets $L_j$ span subtrees that intersect only at the root, 
  and for each $j$ 
  there is a $2\times 2^{|L_j|}$ matrix $A_j$, dependent only on the Markov matrices 
  on edges of the subtrees, which expresses the joint probabilities
  of states at the leaves in $L_j$, conditioned on the root. In particular, $A_1=M_e$. 
  Denoting the rows of $A_j$ by $a_j,b_j$, the $r$-dimensional flattening of our distribution is
$$ \pi_1 \cdot a_1 \otimes a_2 \otimes \cdots \otimes a_r
\,\,+ \,\,  \pi_2 \cdot b_1 \otimes b_2 \otimes \cdots \otimes b_r. $$
If  $a_{1i}=0$ (or $b_{1i}=0$) then the $i$th slice in the first index
has rank $\leq 1$.
 
  For an internal edge $e$ as in (2), assume the root of the tree is located at
  one end, and the Markov matrix on the edge is $M_e$, with rows $a_e,b_e$. 
  Let $L_1$, $L_2$ and $L_3$, $L_4$ be the leaves of the subtrees 
  attached to the respective ends of $e$.
  Then the $2^{|L_1\cup L_3|}\times 2^{|L_2\cup L_4|}$
  matrix flattening incompatible with $T$ can be expressed as
  \begin{equation}
  \label{eq:middleflat}
    A^T \operatorname{diag}(\pi_1a_e,\pi_2 b_e)B,
    \end{equation}
  where $A,B$ are $4\times 2^{|L_1\cup L_3|}$ and $4\times 2^{|L_2\cup
  L_4|}$ matrices, respectively. The entries of $A$ depend on the
  parameters on the subtrees on $L_1$ and $L_3$, while those of
  $B$ depend on the parameters on the subtrees on $L_2$ and
  $L_4$. Thus if $M_e$ has a zero entry, 
  then the $4 \times 4$-matrix 
$  \operatorname{diag}(\pi_1a_e,\pi_2 b_e)$ is singular, and 
hence   the flattening (\ref{eq:middleflat}) has rank at most $3$.
\end{proof}

Several recent works found semialgebraic descriptions of the 2-state
general Markov model on trees that is considered here.  In
\cite{pwz-2009-semialgebraictrees} a different coordinate system is
used, but \cite{allman2012semialgebraic} follows the same framework as
this paper. Although some of the inequalities given in
\cite{allman2012semialgebraic} hint at the form of the algebraic
boundary determined in Theorem \ref{thm:tree1}, those inequalities are
considerably more complicated than our description here.  While the
inequalities provide tests for model membership, the relative
simplicity of the algebraic boundary is expected to be advantageous
for other purposes, such as understanding the geometry of
log-likelihood functions over $\mathcal{M}_T$, and studying the limit
behavior of iterative methods for parameter estimation such as
Expectation Maximization (EM).

\section{Towards Higher Rank}

There are formidable obstacles to extending
our results to tensors of rank $r>2$.  First of all, we do not know
how to generalize the supermodular constraints.  Second, we run into
problems of non-identifiability, even in the case of matrices
$(n=2)$. Recall also (e.g.~from \cite[Example 4.1.2]{LAS}) that a
nonnegative matrix of rank $3$ need not have nonnegative rank $3$.
The topological analysis given by Mond {\it et al.}~\cite{MSS} illustrates well the
difficulties involved in obtaining a characterization of the
semialgebraic set of $d_1 \times d_2$ matrices of nonnegative rank $\leq 3$.

On the other hand, for tensors of dimension $n \geq 3$, rank decompositions
are often identifiable when $r$ is small relative to $d_1,d_2,\ldots,d_n$.
In such situations, when the model is identifiable, 
one might hope for results similar to
Theorems \ref{theorem:zwei} and  \ref{thm:tree1}.
However, a third obstacle arises: in order to characterize 
algebraic boundaries, one needs a version of Lemma \ref{lem:sing} for
the singular locus
$\Theta_{\rm sing}$ of the model parameterization~$\phi$.

\smallskip

In what follows, we illustrate these issues for two rank $3$ examples.

\begin{exmp}
\label{ex:332rank3}
Consider the set $\mathcal{M}$ of $3\times 3\times 2$ tensors of
nonnegative rank $\leq 3$.  This is a smallest format for which
rank $3$ decompositions are generically unique, up to label
swapping. Normalizing the tensor entries to sum to $1$, we obtain
$\mathcal{M}$ as the image of the map
 $$ \begin{matrix} \phi : \Theta \rightarrow \Delta_{17} \,,\,\
   ( \pi;a_1,a_2,a_3;b_1,b_2,b_3;c_1,c_2,c_3)\mapsto \qquad \\
   \qquad\quad \pi_1 \cdot a_1\otimes a_2\otimes a_3 +\pi_2 \cdot
   b_1\otimes b_2\otimes b _3 +\pi_3 \cdot c_1\otimes c_2\otimes c_3,
 \end{matrix}
 $$
 where $\pi\in \R_{\ge 0}^3$, $a_i,b_i,c_i\in \R_{\ge 0}^3$ for
 $i=1,2$, and $a_3,b_3,c_3\in \R_{\ge 0}^2$~all have coordinate sum
 $1$.  The domain is the polytope $\Theta = \Delta_2 \times (\Delta_2
 {\times} \Delta_2 {\times} \Delta_1)^3$.  The facets of $\Theta$ are
 given by parameters being $0$.  This map is generically $6$-to-$1$,
 so $\phi(\Theta) = \mathcal{M}$ is full-dimensional in the simplex
 $\Delta_{17}$.  The Zariski closure $\overline{\mathcal{M}}$ is the
 entire projective space $\mathbb{P}^{17} $ of $3 {\times} 3 {\times}
 2$ tensors.

The algebraic boundary $\overline{\partial \mathcal{M}}$ has eight
irreducible components:

\begin{enumerate}
\item[(a)]
Two components $\,\overline{\phi(\{a_{3k} = 0\})} \, = 
\overline{\phi(\{b_{3k} = 0\})} \, = 
\overline{\phi(\{c_{3k} = 0\})} $, for $k=1,2$,
given by the $3 {\times} 3$-slice $P_k = [p_{**k}]$ having rank $\leq 2$.
\item[(b)] Three components given, for $i=1,2,3$,
 by the $3 \times 3$-matrix $P_1 \cdot (P_2)^{-1}$
having an eigenvector with zero $i$-th coordinate.
\item[(c)] Three components given, for $j =1,2,3$,
by the $3 \times 3$-matrix $P_1^T  \cdot (P_2)^{-T}$
having an eigenvector with zero $j$-th coordinate.
\end{enumerate}
The two components (a) are the cubic hypersurfaces given by the
determinants of $P_1$ and $P_2$.
The six components (b) and (c) are hypersurfaces
of degree $6$.
For instance,  the polynomial $K $ that defines the (b)
component $\,\overline{\phi(\{a_{13} = 0\})}  = 
\overline{\phi(\{b_{13} = 0\})} = 
\overline{\phi(\{c_{13} = 0\})} $
~equals
\begin{small}
$$ \begin{matrix} K\,=\,
  p_{111} p_{212} p_{321}^2 p_{332}^2 
- 2 p_{111} p_{212} p_{321} p_{322} p_{331} p_{332} + p_{111} p_{212} p_{322}^2 p_{331}^2 \qquad \\
\quad
 - p_{111} p_{222} p_{311} p_{321} p_{332}^2 
 + p_{111} p_{222} p_{311} p_{322} p_{331} p_{332}
   - p_{111} p_{222} p_{312} p_{322} p_{331}^2 \\
 + p_{111} p_{222} p_{312} p_{321} p_{331} p_{332} 
  + p_{111} p_{232} p_{311} p_{321} p_{322} p_{332} 
   - p_{112} p_{211} p_{321}^2 p_{332}^2 \\ 
   - p_{111} p_{232} p_{311} p_{322}^2 p_{331} 
  + p_{111} p_{232} p_{312} p_{321} p_{322} p_{331}
     - p_{111} p_{232} p_{312} p_{321}^2 p_{332}  \\
 + 2 p_{112} p_{211} p_{321} p_{322} p_{331} p_{332} 
 - p_{112} p_{211} p_{322}^2 p_{331}^2 + p_{112} p_{221} p_{311} p_{321} p_{332}^2 \\
 - p_{112} p_{221} p_{311} p_{322} p_{331} p_{332} 
  - p_{112} p_{221} p_{312} p_{321} p_{331} p_{332}
 + p_{112} p_{221} p_{312} p_{322} p_{331}^2 \\ 
  - p_{112} p_{231} p_{311} p_{321} p_{322} p_{332} 
 + p_{112} p_{231} p_{311} p_{322}^2 p_{331}  + p_{112} p_{231} p_{312} p_{321}^2 p_{332} \\
 - p_{112} p_{231} p_{312} p_{321} p_{322} p_{331} 
  - p_{121} p_{212} p_{311} p_{321} p_{332}^2
 + p_{121} p_{212} p_{311} p_{322} p_{331} p_{332} \\
   + p_{121} p_{212} p_{312} p_{321} p_{331} p_{332}
 - p_{121} p_{212} p_{312} p_{322} p_{331}^2 + p_{121} p_{222} p_{311}^2 p_{332}^2 \\ 
  \end{matrix}  $$  $$ \begin{matrix}
  - 2 p_{121} p_{222} p_{311} p_{312} p_{331} p_{332}
 + p_{121} p_{222} p_{312}^2 p_{331}^2
 - p_{121} p_{232} p_{311}^2 p_{322} p_{332} \\ 
      + p_{121} p_{232} p_{311} p_{312} p_{321} p_{332}
 + p_{121} p_{232} p_{311} p_{312} p_{322} p_{331} - p_{121} p_{232} p_{312}^2 p_{321} p_{331} \\ 
 + p_{122} p_{211} p_{311} p_{321} p_{332}^2 
 - p_{122} p_{211} p_{311} p_{322} p_{331} p_{332}
 - p_{122} p_{211} p_{312} p_{321} p_{331} p_{332} \\ 
  + p_{122} p_{211} p_{312} p_{322} p_{331}^2 
 - p_{122} p_{221} p_{311}^2 p_{332}^2 + 2 p_{122} p_{221} p_{311} p_{312} p_{331} p_{332} \\
 - p_{122} p_{221} p_{312}^2 p_{331}^2   + p_{122} p_{231} p_{311}^2 p_{322} p_{332}
 - p_{122} p_{231} p_{311} p_{312} p_{321} p_{332} \\ 
    - p_{122} p_{231} p_{311} p_{312} p_{322} p_{331} 
 + p_{122} p_{231} p_{312}^2 p_{321} p_{331} + p_{131} p_{212} p_{311} p_{321} p_{322} p_{332} \\
 - p_{131} p_{212} p_{311} p_{322}^2 p_{331}  - p_{131} p_{212} p_{312} p_{321}^2 p_{332}
 + p_{131} p_{212} p_{312} p_{321} p_{322} p_{331} \\
  - p_{131} p_{222} p_{311}^2 p_{322} p_{332} 
 + p_{131} p_{222} p_{311} p_{312} p_{321} p_{332} + p_{131} p_{222} p_{311} p_{312} p_{322} p_{331} \\
 - p_{131} p_{222} p_{312}^2 p_{321} p_{331} + p_{131} p_{232} p_{311}^2 p_{322}^2
 - 2 p_{131} p_{232} p_{311} p_{312} p_{321} p_{322} \\ 
 + p_{131} p_{232} p_{312}^2 p_{321}^2 
 - p_{132} p_{211} p_{311} p_{321} p_{322} p_{332} + p_{132} p_{211} p_{311} p_{322}^2 p_{331} \\
 + p_{132} p_{211} p_{312} p_{321}^2 p_{332}  - p_{132} p_{211} p_{312} p_{321} p_{322} p_{331} 
 + p_{132} p_{221} p_{311}^2 p_{322} p_{332} \\ - p_{132} p_{221} p_{311} p_{312} p_{321} p_{332} 
 - p_{132} p_{221} p_{311} p_{312} p_{322} p_{331}  + p_{132} p_{221} p_{312}^2 p_{321} p_{331} \\
 - p_{132} p_{231} p_{311}^2 p_{322}^2 + 2 p_{132} p_{231} p_{311} p_{312} p_{321} p_{322} 
- p_{132} p_{231} p_{312}^2 p_{321}^2.
\end{matrix}
$$
\end{small}

This polynomial was found using
the reduced Kalman matrix in \cite[equation (1.5)]{OS}.
Under the parametrization, this expression factors as
\begin{multline*}
K = 
\pi_1^2\pi_2^2\pi_3^2\, a_{13} b_{13} c_{13}
 (a_{31} b_{32} - a_{32} b_{31})
 (a_{31} c_{32} - a_{32} c_{31})
  (b_{31} c_{32} - b_{32} c_{31})\\
  \times
  {\rm det} [a_1,b_1,c_1]
{\rm det}[a_2,b_2,c_2]^2
\end{multline*}

To prove that there is nothing else in $\overline{\partial \mathcal{M}}$,
we proceed as in Lemma \ref{lem:fiber}.
We examine the Jacobian of  $\phi$, which has
rank $17$ at generic points of $\Theta$. Using symbolic computation, we find that
its singular locus
$\Theta_{\rm sing}$, where the rank drops, decomposes into
three types of components:
\begin{enumerate}
\item points with $\pi$ having a zero entry, at which the rank of the
  Jacobian is generically 12,
\item points with two of $a_3,b_3,c_3$ equal, at which the rank of the
  Jacobian is generically 15,
\item points with $a_1,b_1,c_1$ (or with $a_2,b_2,c_2$) linearly
  dependent, at which the rank of the Jacobian is generically 14.
\end{enumerate}
We now show that every point of $\Theta_{\rm sing}$ lies in a fiber of
$\phi$ that intersects the boundary of the polytope $\Theta$.  For
singular points of type (1), if say $\pi_1=0$, one may replace an
$a_i$ with any other vector to obtain another point in the fiber, so
this is clear. For type (2), if say $b_3=c_3$, then
$$\pi_2 \cdot b_1\otimes b_2\otimes b _3\, +\,
\pi_3 \cdot c_1\otimes c_2\otimes c_3\,\,= \,\,A\otimes b_3$$ for a
$3\times 3$ matrix $A$ of nonnegative rank $2$. Since one can find a
nonnegative rank $2$ decomposition of $A$ with zeros in some vector
entry, we can construct the desired boundary point in the fiber.

For type (3) singular points we argue as follows.  Suppose $P$ is the
image of parameters where $a_2,b_2,c_2$ are dependent. Let $d$ be a
nonzero, nonnegative vector in the span of $a_1,b_1$ and consider the
line of tensors
\begin{align*}B(t)\,&=\,P-\pi_3 \cdot (c_1+td)\otimes c_2\otimes c_3\\
&=\, a_1\otimes a_2\otimes a_3 \,+\,
\pi_2 \cdot b_1\otimes b_2\otimes b _3 \,-\, t\pi_3 \cdot d\otimes c_2\otimes c_3.
\end{align*}
The sets
$\{a_1,b_1,d\}$, $\{a_2,b_2,c_2\}$ and $\{a_3,b_3,c_3\}$ are
dependent, so the three matrix flattenings of $B(t)$ have rank $\leq 2$.
Thus for all $t$, the tensor $B(t)$ has border rank $\leq 2$. 
Also, $B(0)$ has nonnegative rank~$2$.

Since $B(t)$ fails to have nonnegative rank~$2$ for $t \gg 0$, there
exists $t_0 \geq 0$ such that $B(t_0)$ lies on the boundary of the
tensors of nonnegative rank $2$.  By Theorem \ref{theorem:zwei},
$B(t_0)$ has a nonnegative rank $2$ decomposition with a zero
coordinate in some parameter vector.  Since
 $$ P\,=\, B(t_0)+\pi_3(c_1+t_0d)\otimes c_2\otimes c_3,$$
 the tensor $P$ has a nonnegative rank $3$ decomposition with a zero
 parameter. Hence $P$ lies in one of the
 eight varieties seen in (a),(b),(c).
 
 We note that a distribution $P$ with invertible slices $P_i$
   lies in $\mathcal{M}$ if and
 only if all eigenvalues and eigenvectors of the $3 \times 3$-matrices
 $P_1 \cdot (P_2)^{-1}$ and $P_1^T \cdot (P_2)^{-T}$ are
 nonnegative. Here, one should be able to pass to the closure and infer
  a nice semialgebraic description of $\mathcal{M}$. \qed
 \end{exmp}

\begin{exmp}
\label{ex:2222rank3}
Let $\mathcal{M}$ be the set of $2 \times 2\times 2\times 2$ tensors
of nonnegative rank $\leq 3$. As in the previous example, we normalize tensors to have entries summing to one. This model is not identifiable: the
generic fiber of its stochastic parametrization $\phi$ is a curve.
Facets of the parameter polytope $\Theta = \Delta_2 \times
(\Delta_1)^{12}$ are mapped into subsets of the model $\mathcal{M}$
that are Zariski dense in $\overline{\mathcal{M}}$.  We note that the
$13$-dimensional variety $\overline{\mathcal{M}}$ is a complete
intersection of degree $16 = 4 \cdot 4$ in $\mathbb{P}^{15}$. It is
defined by the determinants of any two of the three $4 \times
4$-flattenings of $P$.

Components of the algebraic boundary $\overline{\partial \mathcal{M}}$
might now be obtained from codimension $2$ faces of the polytope $\Theta$.
For instance, write
$$ P \, = \, \pi_1a_1 \otimes a_2 \otimes a_3 \otimes a_4 \,
+  \pi_2 b_1 \otimes b_2 \otimes b_3 \otimes b_4 \,
+   \pi_3 c_1 \otimes c_2 \otimes c_3 \otimes c_4 ,
$$
and consider the face 
$\{a_{11} = b_{22} = 0\}$ of $\Theta$.
Then $\overline{\phi(\{a_{11} = b_{22} = 0\})}$
is suspected to be a component in 
$\overline{\partial \mathcal{M}}$. This variety has dimension $12$ and
degree $56$ in $\mathbb{P}^{15}$.  It is defined, as a subscheme of
$\overline{\mathcal{M}}$, by $55$ polynomials of degree $8$
in the $16$ unknowns $p_{ijkl}$.
The smallest of these degree $8$ polynomials has $96$ terms,
and we shall resist the temptation to  list them.
The semialgebraic geometry of $\mathcal{M}$ deserves further analysis.
 \qed
\end{exmp}


\begin{thebibliography}{1}

\bibitem{daykin78}
{\sc R.~Ahlswede and D.~Daykin}, {\em An inequality for the weights of two
  families of sets, their unions and intersections}, Z. Wahrscheinlichkeitstheorie und
  Verwandte Gebiete
 {\bf 43} (1978) 183--185.

\bibitem{allmanrhodesold}
{\sc E.~Allman and J.~Rhodes},
 {\em Phylogenetic ideals and varieties for the general Markov model},
  Adv. in Appl. Math. {\bf 40} (2008) 127--148. 
 
\bibitem{allman2012semialgebraic}
{\sc E.~S. Allman, J.~A. Rhodes, and A.~Taylor}, {\em A semialgebraic
  description of the general Markov model on phylogenetic trees}, 
  preprint, {\tt arXiv:1212.1200}.
  
\bibitem{altmannsquare}
{\sc D.~Altmann and K.~Altmann},
{\em Estimating vaccine coverage by computer algebra},
IMA J.~Mathem.~Applied in Medicine and Biology {\bf 17} (2000) 137--146.

\bibitem{de2008tensor}
{\sc V.~de~Silva and L.~Lim}, {\em {Tensor rank and the ill-posedness of the
  best low-rank approximation problem}}, SIAM J. Matrix Anal. Appl, {\bf 30} (2008) 1084--1127.

\bibitem{draismakuttler}
{\sc J.~Draisma and J.~Kuttler}, {\em On the ideals of equivariant tree models},
 Mathematische Annalen {\bf 344} (2009) 619--644.

\bibitem{LAS}
{\sc M.~Drton, B.~Sturmfels, and S.~Sullivant},
Lectures on Algebraic Statistics, Oberwolfach Seminars, 39. Birkh\"auser Verlag, Basel, 2009.

\bibitem{EHS}
{\sc A.~Engstr\"om, P.~Hersh and B.~Sturmfels},
{\em  Toric cubes}, Circolo Matematico di Palermo {\bf 62} (2013) 67--78. 
  
\bibitem{FriedlanderHatz}
{\sc M.~Friedlander and K.~Hatz}, {\em Computing nonnegative tensor factorizations},
Optimization Methods and Software {\bf 23} (2008) 631--647.


\bibitem{GSS} {\sc L.~Garcia, M.~Stillman, and B.~Sturmfels},
{\em Algebraic geometry of Bayesian networks}, J.~Symbolic Computation
{\bf 39} (2005) 331--355.

\bibitem{M2} {\sc D. Grayson and M. Stillman},
{\em Macaulay 2, a software system for research in algebraic geometry}.
 Available at {\tt http://www.math.uiuc.edu/Macaulay2/}.

\bibitem{klaeretripod}
{\sc S.~Klaere and V.~Liebscher}, {\em An algebraic analysis of the two state
  {M}arkov model on tripod trees}, Math. Biosciences {\bf 237} (2012) 38--48.

\bibitem{JM} {\sc J.M.~Landsberg},
{\em Tensors: Geometry and Applications}, Graduate Studies in Mathematics, 128. American Mathematical Society, Providence, RI, 2012.
  
\bibitem{LM} {\sc J.M.~Landsberg and L.~Manivel}, {\em On ideals of secant varieties of
Segre varieties}, Found.~Comput.~Math. {\bf 4} (2004) 397--422.

\bibitem{MSS}
{\sc D.~Mond, J.Q.~Smith, and D.~van Straten},
{\em Stochastic factorizations, sandwiched simplices and the topology of the space of explanations},
R.~Soc.~Lond.~Proc. Ser.~A Math.~Phys.~Eng.~Sci. {\bf 459} (2003) 2821--2845.

\bibitem{CRT} {\sc J.Morton, L.Pachter, A.Shiu, B.Sturmfels, and O.Wienand},
{\em Convex rank tests and semigraphoids},
SIAM J.~Discrete Math.~{\bf 23} (2009) 1117--1134.

\bibitem{pearl_tarsi86}
{\sc J.~Pearl and M.~Tarsi}, {\em Structuring causal trees}, J. Complexity {\bf 2}
  (1986) 60--77.
\newblock Complexity of approximately solved problems (Morningside Heights,
  1985).

\bibitem{raicu}
{\sc C.~Raicu}, {\em  Secant varieties of Segre--Veronese varieties}, 
Algebra and Number Theory 
{\bf 6} (2012) 1817--1868.

\bibitem{OS}
{\sc G.~Ottaviani and B.~Sturmfels},
{\em Matrices with eigenvectors in a given subspace},
Proceedings  American Math.~Society {\bf 141} (2013) 1219--1232. 

\bibitem{steelfaller}
{\sc M.~Steel and B.~Faller}, {\em Markovian log-supermodularity, and
its applications in phylogenetics},
Applied Mathematics Letters {\bf 22} (2009) 1141--1144.

\bibitem{pwz-2009-semialgebraictrees}
{\sc P.~Zwiernik and J.~Q. Smith}, {\em Implicit inequality constraints in a
  binary tree model}, Electron. J. Statist. {\bf 5} (2011) 1276--1312.


\end{thebibliography}

\end{document}